\documentclass[12pt,reqno]{article}

\usepackage[usenames]{color}
\usepackage{amssymb}
\usepackage{amsmath}
\usepackage{amsthm}
\usepackage{amsfonts}
\usepackage{amscd}
\usepackage{graphicx}

\usepackage[colorlinks=true,
linkcolor=webgreen,
filecolor=webbrown,
citecolor=webgreen]{hyperref}

\definecolor{webgreen}{rgb}{0,.5,0}
\definecolor{webbrown}{rgb}{.6,0,0}

\usepackage{color}
\usepackage{fullpage}
\usepackage{float}

\usepackage{graphics}
\usepackage{latexsym}
\usepackage{epsf}
\usepackage{breakurl}

\newcommand{\seqnum}[1]{\href{https://oeis.org/#1}{\rm \underline{#1}}}
\def\suchthat{\,:\,}

\def\modd#1 #2{#1\ \mbox{\rm (mod}\ #2\mbox{\rm )}}
\def\Enn{{\mathbb{N}}}

\begin{document}


\theoremstyle{plain}
\newtheorem{theorem}{Theorem}
\newtheorem{corollary}[theorem]{Corollary}
\newtheorem{lemma}[theorem]{Lemma}
\newtheorem{proposition}[theorem]{Proposition}

\theoremstyle{definition}
\newtheorem{definition}[theorem]{Definition}
\newtheorem{example}[theorem]{Example}
\newtheorem{conjecture}[theorem]{Conjecture}

\theoremstyle{remark}
\newtheorem{remark}[theorem]{Remark}

\title{Some Fibonacci-Related Sequences}

\author{Benoit Cloitre\\
France \\
\href{mailto:benoit7848c@yahoo.fr}{\tt benoit7848c@yahoo.fr}\\
\and 
Jeffrey Shallit\\
School of Computer Science \\
University of Waterloo \\
Waterloo, ON N2L 3G1 \\
Canada\\
\href{mailto:shallit@uwaterloo.ca}{\tt shallit@uwaterloo.ca}
}

\maketitle

\begin{abstract}
We discuss an interesting sequence defined recursively;
namely, sequence \seqnum{A105774} from the
{\it On-Line Encyclopedia of Integer Sequences}, and study some
of its properties.   Our main tools are Fibonacci representation,
finite automata, and the {\tt Walnut} theorem-prover.
We also prove two new results about synchronized sequences.
\end{abstract}

\section{Introduction}

Define the Fibonacci numbers $(F_n)_{n \geq 0}$ as usual by 
the initial values $F_0 = 0$, $F_1 = 1$, and $F_n = F_{n-1} + F_{n-2}$
for $n \geq 2$.   In 2005, the first author defined an interesting sequence,
\seqnum{A105774}, $(a(n))_{n\geq 0}$ as follows:
\begin{equation}
a(n) = \begin{cases}
	n, & \text{if $n\leq 1$}; \\
	F_{j+1} - a(n-F_j), & \text{if $F_j < n \leq F_{j+1}$ for $j\geq 2$}.
	\end{cases}
	\label{eq1}
\end{equation}
The first few values of $(a(n))$ are given in Table~\ref{tab1}.
Although strictly speaking the sequence was originally defined only
for $n \geq 1$, it makes sense to set $a(0)=0$, which we have done.
\begin{table}[H]
\begin{center}
\begin{tabular}{c|ccccccccccccccccccccc}
$n$ & 0& 1& 2& 3& 4& 5& 6& 7& 8& 9&10&11&12&13&14&15&16&17&18&19&20 \\
\hline
$a(n)$ & 0& 1& 1& 2& 4& 4& 7& 7& 6&12&12&11& 9& 9&20&20&19&17&17&14&14
\end{tabular}
\end{center}
\label{tab1}
\end{table}
The sequence has an intricate fractal structure, which is depicted
in Figure~\ref{fig22}.
\begin{figure}[htb]
\begin{center}
\includegraphics[width=6.5in]{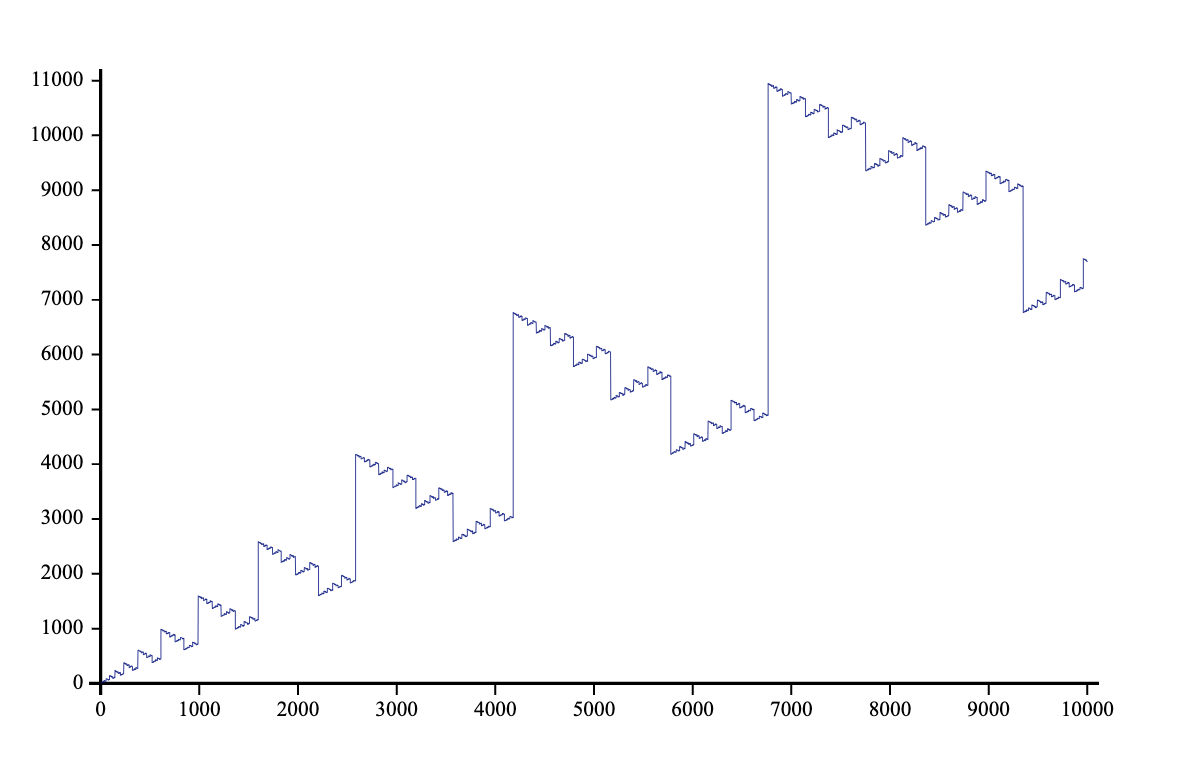}
\end{center}
\caption{Graph of $a(n)$.}
\label{fig22}
\end{figure}

A number of properties of this sequence were stated, without proof, in the
index entry for \seqnum{A105774} in the {\it On-Line Encyclopedia of
Integer Sequences} (OEIS) \cite{Sloane:2023}.
In this paper we prove these properties,
and many new ones, with the aid of finite automata and
the {\tt Walnut} theorem-prover
\cite{Mousavi:2016,Shallit:2022}.

We will need the Lucas numbers, $L_n$, defined
by $L_0 = 2$, $L_1 = 1$, and $L_n = L_{n-1} + L_{n-2}$ for $n \geq 2$.
Recall the Binet forms:
\begin{align}
F_n &= (\varphi^n - (-1/\varphi)^n)/\sqrt{5} \\
L_n &= \varphi^n + (-1/\varphi)^n .
\end{align}

In addition to studying \seqnum{A105774}, we also prove two new
results about synchronized sequences, in Theorems~\ref{new1} and
\ref{new2}.

\section{Fibonacci representations}

In this paper we represent the natural 
numbers $\Enn = \{ 0,1,2,\ldots \}$ in the Fibonacci (aka Zeckendorf)
numeration system \cite{Lekkerkerker:1952,Zeckendorf:1972}.
In this system we write $n$ as the sum of
distinct Fibonacci numbers $F_i$ for $i \geq 2$, subject to the
condition that no two consecutive Fibonacci numbers appear in the sum.
For example, $43 = F_9 + F_6 + F_2 = 34+8+1$.   Usually we write
this representation as a bit string $(n)_F = e_1 e_2 \cdots e_t$,
with $e_i \in \{0,1 \}$, $e_1 \not= 0$, and 
$e_i e_{i+1} \not=1$, such that $n = \sum_{1 \leq i \leq t} e_i F_{t-i+2}$.
For example, $(43)_F = 10010001$.   The inverse function, mapping
an arbitrary bit string (with no conditions) $x = e_1 e_2 \cdots e_t$
to the number it represents is defined
to be $[x]_F = \sum_{1 \leq i \leq t} e_i F_{t-i+2}$.

\section{Finite automata and regular expressions}

A finite automaton is a simple model of a computer.  In its simplest
incarnation, a {\it deterministic finite automaton\/} (DFA),
the automaton can be
in any one of a finite number of different states.   It takes
as input a finite string over a finite alphabet, and processes its
input symbol-by-symbol, choosing the next state through a lookup
table (the {\it transition function\/}) 
based only on the current state and the next input symbol.
Some of the states are designated as {\it accepting states}; an
input string is accepted if the automaton is in an accepting state after all 
its symbols are processed.  In this paper we represent integers in
Fibonacci representation, so we can think of the DFA as processing integers
instead of strings.   Thus, a DFA is a good model to represent a {\it set\/}
of natural numbers; namely, the set of all integers whose representations
are accepted.  In this sense, a DFA decides {\it membership\/} for
a set (or sequence).  DFA's are typically drawn as {\it transition diagrams},
where the arrows represent transitions between states, double circles
denote accepting states, and single circles denote nonaccepting states.

As an example, consider the automaton in Figure~\ref{fig0}.   It takes
the Fibonacci representation of a natural number $n$ as input, and accepts
if and only if $n$ is even.
\begin{figure}[H]
\begin{center}
\includegraphics[width=6.5in]{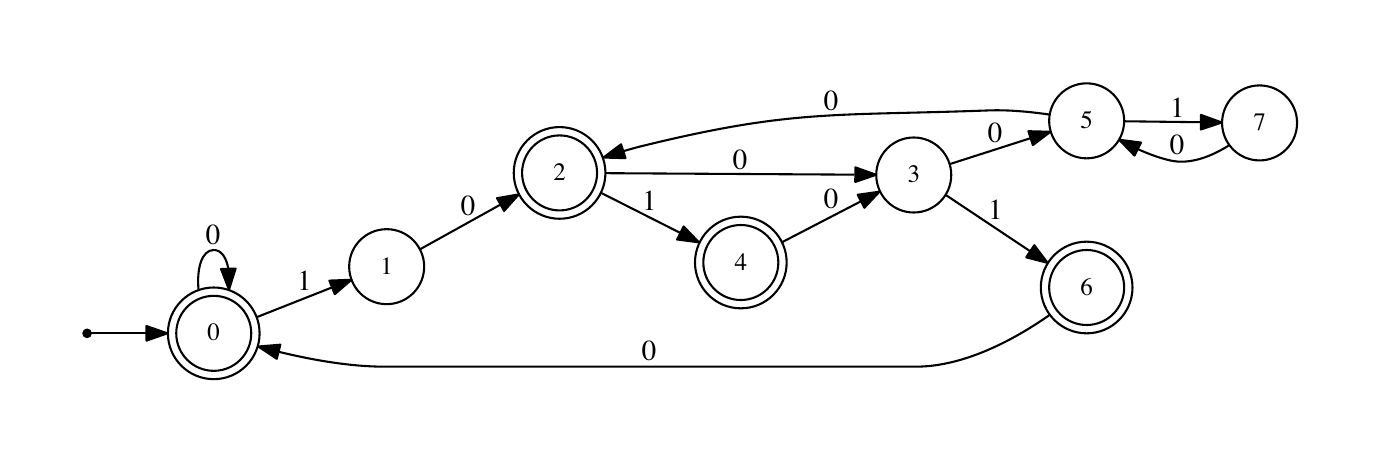}
\end{center}
\caption{Fibonacci automaton for the set of even numbers.}
\label{fig0}
\end{figure}

A slightly more complicated model is the {\it deterministic finite automaton
with output\/} (DFAO).  In this model, an output is associated with every
state, and the output on an input string $x$ is the output associated with
the last state reached after processing $x$.  This model is adequate for
representing sequences over a finite alphabet; since there are only finitely
many states, only finitely many different outputs are possible.
If a sequence $(a(n))$ is computed by a DFAO in the following way---the input
is the representation of $n$ and the output is $a(n)$---then we say
$(a(n))$ is {\it automatic}.  Again, a DFAO is often displayed as a transition
diagram, with the notation $q/i$ in a state meaning that the state is
numbered $q$ and has output $i$. 

For example, the transition diagram in Figure~\ref{fig15} illustrates
a DFAO computing the sequence $(n \bmod 3)_{n \geq 0}$, where the input
is in Fibonacci representation.
\begin{figure}[H]
\begin{center}
\includegraphics[width=6.5in]{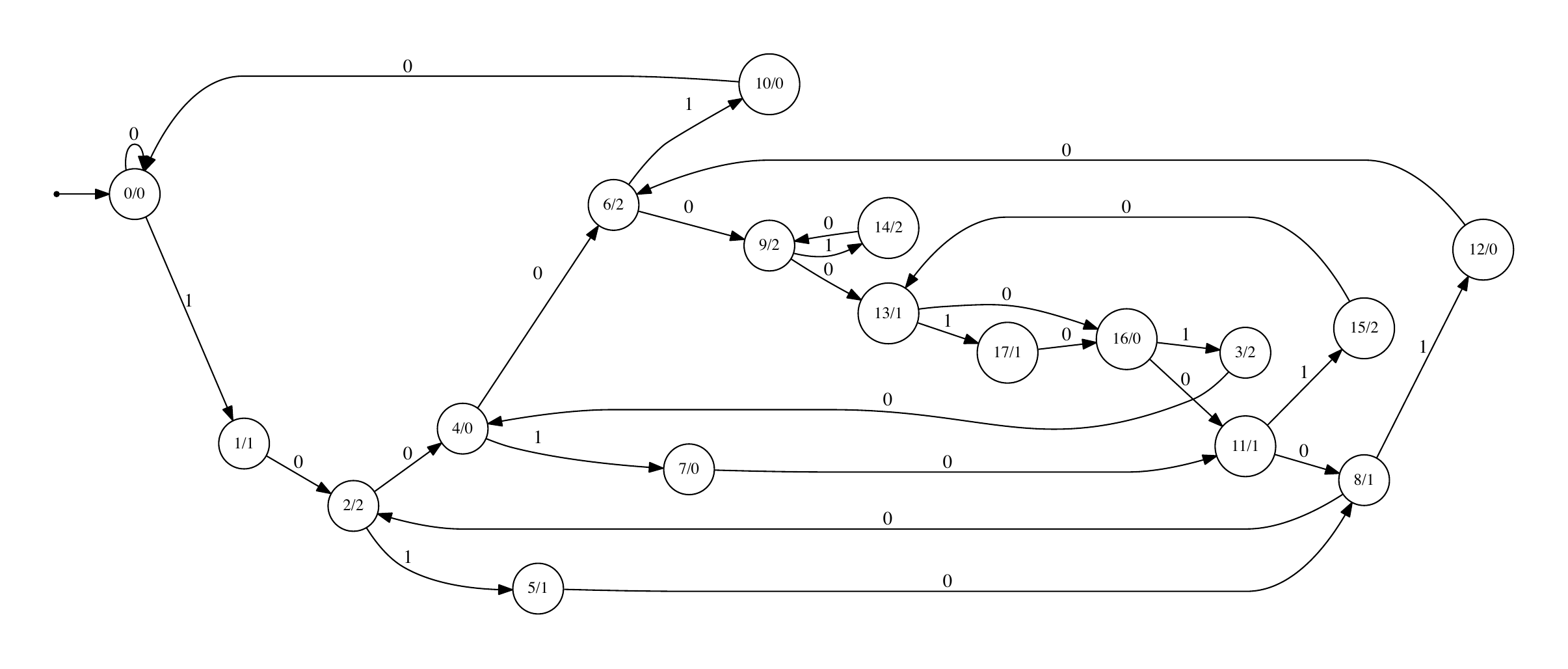}
\end{center}
\caption{Fibonacci DFAO computing $n \bmod 3$.}
\label{fig15}
\end{figure}
We note in passing that, due to a nice result of
Charlier et al.~\cite{Charlier},
the minimal DFAO computing
$(n \bmod k)_{n \geq 0}$ for $n$ in Fibonacci representation is
known to have $2k^2$ states.

Finally, there is the notion of {\it synchronized automaton\/} that computes
a function $f:\Enn \rightarrow \Enn$.  In this case, we use a DFA, but
with {\it two\/} integer inputs $n$ and $x$ read in parallel,
where the shorter input is padded with leading zeros, if necessary,
and the DFA accepts if and only if $x = f(n)$.  This model is suited
to representing sequences of natural numbers.  It has the important advantage
that if a sequence is represented in this way, then there is a decision
procedure to answer first-order queries about the values of the sequence,
in the logical structure $\langle \Enn, +, n \rightarrow f(n) \rangle$.
See \cite{Shallit:2021h} for more details.

As an example, the synchronized automaton in Figure~\ref{fig16} computes
the function $n \rightarrow \lfloor n/2 \rfloor$ in Fibonacci representation.
\begin{figure}[H]
\begin{center}
\includegraphics[width=6.5in]{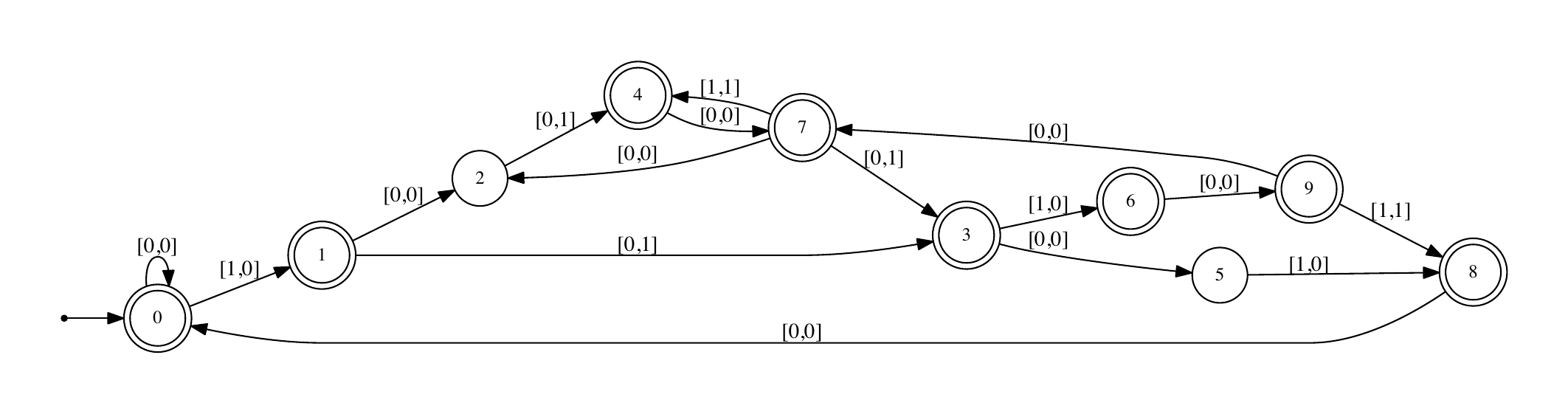}
\end{center}
\caption{Synchronized Fibonacci automaton computing $\lfloor n/2 \rfloor$.}
\label{fig16}
\end{figure}

It is quite useful to adopt the convention that in all three models, the result
is not sensitive to the presence or absence of leading zeros in the
representation of numbers.

The most powerful of all three representations for a sequence is the synchronized automaton, since one can obtain the other two from it easily.   The downside
is that it is quite possible that there exists a DFA to decide membership
in a sequence $\bf s$, 
but there is no corresponding synchronized
automaton computing the $n$'th term of $\bf s$.

In this paper, we sometimes use the notation called a {\it regular expression}.
For our purposes, the main thing to know
is that $x^*$, where $x$ is a single string or set of strings,
denotes the set of all strings consisting of zero or more 
concatenations of elements chosen from the set $x$ specifies.
For more about automata and regular expressions,
see \cite{Hopcroft&Ullman:1979}.

\section{{\tt Walnut}}

{\tt Walnut} is a theorem-prover for automatic sequences.  It can answer
first-order queries about automatic and synchronized sequences defined
over the natural numbers $\Enn$.
The syntax is basically first-order logic, with the following
meanings of symbols:
\begin{itemize}
\item {\tt \&} is logical AND, {\tt |} is logical OR,
{\tt \char'176} is logical NOT, {\tt =>} is implication, and
{\tt <=>} means IFF.

\item {\tt !=} means $\neq$ and {\tt x/y} denotes integer division, that
is, $\lfloor x/y \rfloor$.

\item {\tt A} is {\tt Walnut}'s representation for the universal quantifier
$\forall$ and {\tt E} is the representation for $\exists$.

\item {\tt reg} defines a regular expression.

\item {\tt def} defines an automaton for later use.

\item {\tt eval} evaluates a first-order logic expression with no free
variables, and returns either {\tt TRUE} or {\tt FALSE}.   

\item {\tt ?msd\_fib} is a bit of jargon telling {\tt Walnut} to represent
all numbers in the Fibonacci numeration system.

\end{itemize}

When {\tt Walnut} returns {\tt TRUE} (resp., {\tt FALSE}) to an {\tt eval}
query, the result is a {\it rigorous proof\/}
that the particular logical statement holds (resp., does not hold).

As examples, here is the {\tt Walnut} code for generating the automata
in Figures~\ref{fig0} and \ref{fig16}:
\begin{verbatim}
def even "?msd_fib Ek n=2*k":
def nover2 "?msd_fib z=n/2":
\end{verbatim}

We will need some {\tt Walnut} code for
the functions $n \rightarrow \lfloor \varphi n \rfloor$ and
$n \rightarrow \lfloor \varphi^2 n \rfloor$, where
$\varphi = (1+\sqrt{5})/2$, the golden ratio.
The synchronized automata for these can be found in \cite{Shallit:2022}, under
the names {\tt phin} and {\tt phi2n}, respectively.

\section{Guessing and checking}

One of the principal tools we use is the ``guess-and-check'' method.   Here we
{\it guess\/} a finite automaton computing a sequence based on empirical data,
and then use {\tt Walnut} to verify that it is correct, usually by checking
that the defining equation holds.  The guessing procedure is
described in \cite[\S 5.7]{Shallit:2022}.

We can carry out this guessing procedure
for the sequence \seqnum{A105774}, which we call 
$(a(n))$.  It produces a synchronized automaton that
computes the sequence in the following sense:  the input is a pair
$(n,x)$ in Fibonacci representation, and the automaton
accepts if and only if $x = a(n)$.  
The automaton {\tt a105774} that we guess is displayed in Figure~\ref{fig2}.
\begin{figure}[H]
\begin{center}
\includegraphics[width=6.5in]{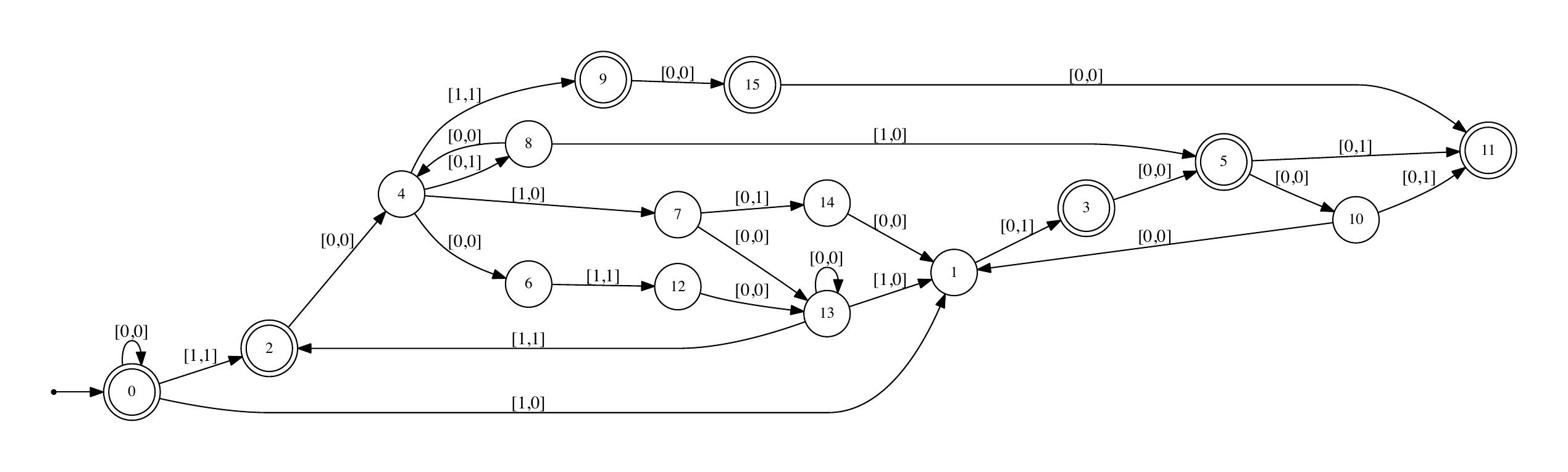}
\end{center}
\caption{Synchronized Fibonacci automaton for $a(n)$.}
\label{fig2}
\end{figure}

The next step, which is crucial, is to
verify that our guessed automaton is actually correct.  First we
must verify that the automaton really does compute a function.
This means that for each $n$ there is exactly one $x$ such that
the pair $(n,x)$ is accepted.  We can verify this as follows:
\begin{verbatim}
eval check_at_least_one "?msd_fib An Ex $a105774(n,x)":
eval check_at_most_one "?msd_fib ~En,x1,x2 x1!=x2 & $a105774(n,x1) &
   $a105774(n,x2)":
\end{verbatim}
and {\tt Walnut} returns {\tt TRUE} for both.

Next we must verify that our automaton obeys
the defining recurrence \eqref{eq1}.
\begin{verbatim}
reg adjfib msd_fib msd_fib "[0,0]*[0,1][1,0][0,0]*":
# accepts (F_k, F_{k+1})
def trapfib "?msd_fib $adjfib(x,y) & x<k & y>=k":
# accepts (k,x,y) if x is the largest Fibonacci number 
# less than k and y is the next largest Fib number
eval test105774 "?msd_fib Ak,x,y,z,t ($trapfib(k,x,y) & $a105774(k,z)
   & $a105774(k-x,t)) => y=z+t":
\end{verbatim}
and {\tt Walnut} returns {\tt TRUE}.  At this point we know that
our guessed automaton is correct.

Let's start by proving a basic property of \seqnum{A105774}.
\begin{proposition}
No natural number appears three or more times in \seqnum{A105774}.
\end{proposition}

\begin{proof}
We use the following {\tt Walnut} code.
\begin{verbatim}
eval test012 "?msd_fib ~Ex,y,z,n x<y & y<z & $a105774(x,n) &
   $a105774(y,n) & $a105774(z,n)":
\end{verbatim}
and {\tt Walnut} returns {\tt TRUE}.
\end{proof}

Now let's create a DFAO computing $c(n)$, the number of times each natural
number appears in \seqnum{A105774}.
We can do this with the following {\tt Walnut} commands:
\begin{verbatim}
def s0 "?msd_fib ~Ex $a105774(x,n)":
def s2 "?msd_fib Ex,y x<y & $a105774(x,n) & $a105774(y,n)":
def s1 "?msd_fib ~($s0(n)|$s2(n))":
combine C s1=1 s2=2 s0=0:
\end{verbatim}
This creates the DFAO in Figure~\ref{fig3}.
\begin{figure}[H]
\begin{center}
\includegraphics[width=6.5in]{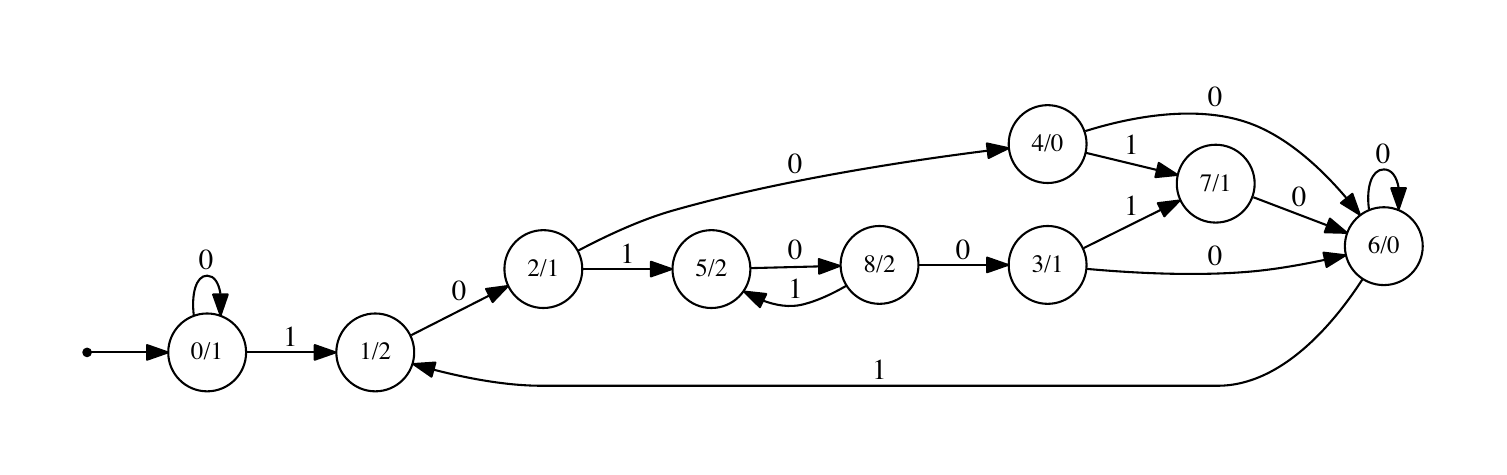}
\end{center}
\caption{Fibonacci DFAO for $c(n)$.}
\label{fig3}
\end{figure}
The first few values of the sequence $c(n)$ are given in 
Table~\ref{tab4}.
\begin{table}[H]
\begin{center}
\begin{tabular}{c|ccccccccccccccccccccc}
$n$ & 0& 1& 2& 3& 4& 5& 6& 7& 8& 9&10&11&12&13&14&15&16&17&18&19&20 \\
\hline
$c(n)$ & 1&2&1&0&2&0&1&2&0&2&0&1&2&0&2&1&0&2&0&1&2 
\end{tabular}
\end{center}
\label{tab4}
\end{table}
It is sequence \seqnum{A368199} in the OEIS.

\begin{proposition}
If $n$ appears twice in \seqnum{A105774}, the two occurrences are
consecutive.
\end{proposition}

\begin{proof}
We use the following {\tt Walnut} code:
\begin{verbatim}
eval twice_consec "?msd_fib An,x,y (x<y & $a105774(x,n) & $a105774(y,n))
   => y=x+1":
\end{verbatim}
and {\tt Walnut} returns {\tt TRUE}.
\end{proof}

The positions $p_0(n)$ of the $n$'th $0$ 
(respectively, the $n$'th $1$ and the $n$'th $2$) in \seqnum{A368199} are also 
Fibonacci-synchronized.  Note that we index starting at $n= 0$.
Once again, we can prove this by guessing and checking.    We omit
the checks that our guessed automata {\tt p0,p1,p2} compute
synchronized functions. 

\begin{verbatim}
eval chek1a "?msd_fib An C[n]=@1 <=> Ek $p1(k,n)":
eval chek2a "?msd_fib An C[n]=@2 <=> Ek $p2(k,n)":
eval chek0b "?msd_fib Aj,m,n ($p0(j,m) & $p0(j+1,n)) => m<n":
eval chek1b "?msd_fib Aj,m,n ($p1(j,m) & $p1(j+1,n)) => m<n":
eval chek2b "?msd_fib Aj,m,n ($p2(j,m) & $p2(j+1,n)) => m<n":
\end{verbatim}
The first few values of $p_0$, $p_1$, and $p_2$ are given in Table~\ref{tab6}.
\begin{table}[H]
\begin{center}
\begin{tabular}{c|ccccccccccccccccccccc}
$n$ & 0& 1& 2& 3& 4& 5& 6& 7& 8& 9&10&11&12&13&14&15&16&17&18 \\
\hline
$p_0(n)$ & 3& 5& 8&10&13&16&18&21&24&26&29&31&34&37&39&42&45&47&50\\
$p_1(n)$ & 0& 2& 6&11&15&19&23&28&32&36&40&44&49&53&57&61&66&70&74\\
$p_2(n)$ & 1& 4& 7& 9&12&14&17&20&22&25&27&30&33&35&38&41&43&46&48\\
\end{tabular}
\end{center}
\label{tab6}
\end{table}
Each of these sequences is already in the OEIS.   We now prove the
characterizations.

\begin{proposition}
The numbers $p_2(n)$ are precisely those in
\seqnum{A007064}, that is, the natural numbers not of the form
$\lfloor \varphi k + {1 \over 2} \rfloor$ for $k \geq 0$.
\label{propp2}
\end{proposition}

\begin{proof}
The sequence \seqnum{A007067} is defined
to be $\lfloor n\varphi + 1/2 \rfloor$. Now 
$$\lfloor n \varphi + 1/2 \rfloor 
= \left\lfloor {{2n\varphi + 1} \over 2} \right\rfloor \\
= \left\lfloor {{ \lfloor 2n\varphi + 1 \rfloor }\over 2} \right\rfloor \\
= \left\lfloor {{ \lfloor 2n\varphi \rfloor + 1} \over 2} \right\rfloor ,
$$
where we have used the fact that 
$\lfloor {{\lfloor x \rfloor } \over n} \rfloor = \lfloor x/n \rfloor$ 
for real numbers $x$ and integers $n \geq 1$.
Hence \seqnum{A007067} can be computed by the following {\tt Walnut} code:
\begin{verbatim}
def a007067 "?msd_fib Ex $phin(2*n,x) & z=(x+1)/2":
\end{verbatim}

The sequence \seqnum{A007064} consists of those integers not in
\seqnum{A007067}.  It is not hard to see that \seqnum{A007064} is given
(after a change of index) by $n \rightarrow \lfloor n\varphi + \varphi/2 \rfloor + n+1$ for $n \geq 0$.  Indeed, we can verify this as follows:
\begin{verbatim}
def a007067 "?msd_fib Ex $phin(2*n,x) & z=(x+1)/2":
def a007064 "?msd_fib Ex $phin(2*n+1,x) & z=n+1+x/2":
eval check_two "?msd_fib Ax (Em $a007067(m,x)) <=> (~En $a007064(n,x))":
\end{verbatim}

Finally we can prove the result about $p_2$ as follows:
\begin{verbatim}
eval checkp2 "?msd_fib An (Ek $p2(k,n)) <=> (Em $a007064(m,n))":
\end{verbatim}
and {\tt Walnut} returns {\tt TRUE}.
\end{proof}

\begin{proposition}
The numbers $p_1(n)$, $n \geq 1$, are precisely those integers in
\seqnum{A035487}.
\end{proposition}

\begin{proof}
We use the following {\tt Walnut} command:
\begin{verbatim}
def a035487 "?msd_fib Ex $a007064(x) & $a007067(x,n)":
eval checkp1 "?msd_fib An (n>0) => ($a035487(n) <=> (Ek $p1(k,n)))":
\end{verbatim}
and {\tt Walnut} returns {\tt TRUE}.
\end{proof}

Finally, we prove the result for $p_0$:
\begin{proposition}
The natural numbers $p_0(n)$, i.e., the numbers
that do not appear in \seqnum{A105774} are
those given by the sequence
\seqnum{A004937}, that is,
$\{ \lfloor \varphi^2 n + {1\over2} \rfloor \suchthat n \geq 1 \}$.
\end{proposition}

\begin{proof}
We have, using the same idea as in the proof of
Proposition~\ref{propp2}, that
$$
\lfloor \varphi^2 n + 1/2 \rfloor 
= \left\lfloor {{{\lfloor \varphi^2 2n \rfloor} + 1} \over 2} \right\rfloor.$$
We can now use {\tt Walnut} to define a synchronized automaton
for \seqnum{A004937}:
\begin{verbatim}
def a004937 "?msd_fib Ex $phi2n(2*n,x) & z=(x+1)/2":
\end{verbatim}
The automaton for \seqnum{A004937} is given in Figure~\ref{fig4}.
\begin{figure}[H]
\begin{center}
\includegraphics[width=6.5in]{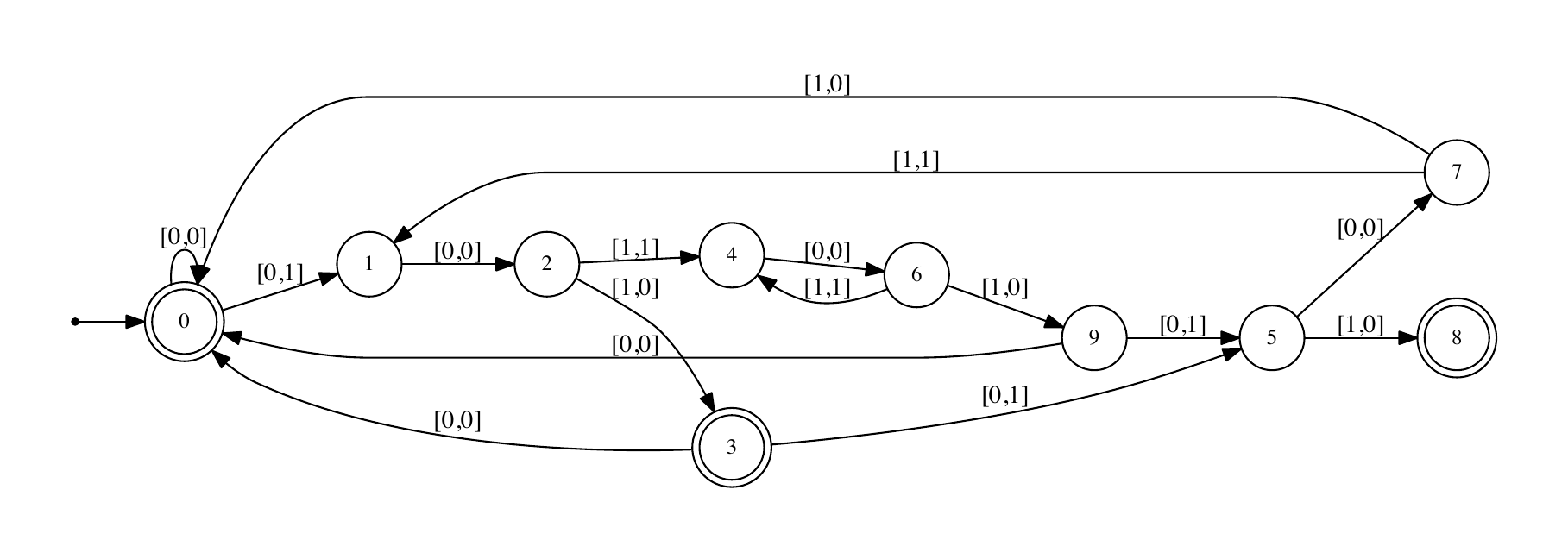}
\end{center}
\caption{Synchronized Fibonacci DFAO for \seqnum{A004937}.}
\label{fig4}
\end{figure}

We can now complete the proof as follows:
\begin{verbatim}
eval chk0 "?msd_fib An (n>0) => ((Ek $a004937(k,n)) <=> (Ej $p0(j,n)))":
\end{verbatim}
and {\tt Walnut} returns {\tt TRUE}.
\end{proof}

\section{Inequalities}

\begin{proposition}
For all $n \geq 0$ we have $\lfloor {{\varphi+2}\over 5}n \rfloor
\leq a(n) \leq \lfloor \varphi n \rfloor$.
\label{prop1}
\end{proposition}

\begin{proof}
For the first inequality we write
\begin{verbatim}
eval lowerbound "?msd_fib An,x,y ($a105774(n,x) & $phin(n,y)) => x>=(y+2*n)/5":
\end{verbatim}
For the second inequality we use the following {\tt Walnut} code:
\begin{verbatim}
eval upperbound "?msd_fib An,x,y ($a105774(n,x) & $phin(n,y)) => x<=y":
\end{verbatim}
and {\tt Walnut} returns {\tt TRUE}.
\end{proof}

The bounds in Proposition~\ref{prop1} are tight, as the following
theorem shows.

\begin{proposition}
We have $\liminf_{n \rightarrow \infty} a(n)/n = {{\varphi+2}\over 5}$
and $\limsup_{n \rightarrow \infty} a(n)/n = \varphi$.
\end{proposition}

\begin{proof}
For the first claim, using the well-known Binet formulas
for the Fibonacci and Lucas numbers, it suffices to show that
$a(L_k +1) = F_{k+1} + 1$ for all $k \geq 3$.
\begin{verbatim}
reg lucfib msd_fib msd_fib "[0,0]*[1,1][0,0][1,0][0,0]*":
# regular expression for the pair (L_k, F_{k+1}) for k>=3
eval chklow "?msd_fib Ax,y $lucfib(x,y) => $a105774(x+1,y+1)":
\end{verbatim}

For the second result it suffices to show that
$a(F_k+1) = F_{k+1} - 1$ for all $k \geq 2$.  This follows directly
from the defining recurrence for $a(n)$.  Or one can use {\tt Walnut}:
\begin{verbatim}
eval chkup "?msd_fib Ax,y,m ($adjfib(x,y) & $a105774(x+1,m)) => m+1=y":
\end{verbatim}
\end{proof}

The graph in Figure~\ref{fig22} suggests studying what the positions
of what might be called ``suffix minima":   those $n$ for which
$a(m)>a(n)$ for all $m > n$.
\begin{theorem}
The suffix minima of $(a(n))$ for $n > 0$ occur precisely when
$(n)_F \in 10(100^*10)^*0^*$.
\end{theorem}

\begin{proof}
We use the following {\tt Walnut} code:
\begin{verbatim}
def suffmin "?msd_fib Am,x,y (m>n & $a105774(m,x) & $a105774(n,y)) => x>y":
\end{verbatim}
The resulting automaton is depicted in Figure~\ref{fig25}, from 
which the regular expression is easily read off.
\begin{figure}[H]
\begin{center}
\includegraphics[width=6.5in]{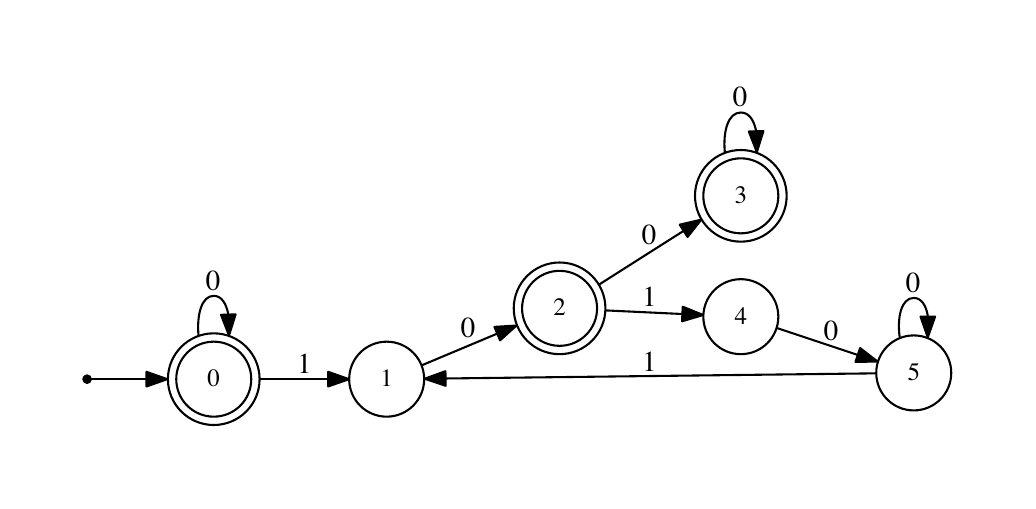}
\end{center}
\caption{Fibonacci representation of positions of suffix minima.}
\label{fig25}
\end{figure}
\end{proof}

\section{Consecutive identical or different terms}

\begin{proposition}
We have 
\begin{itemize}
\item[(a)]  $a(n) = a(n+1)$ if and only if 
$n \in \{ \lfloor \varphi^2 k \rfloor - 1 \suchthat k \geq 1 \}$;
\item[(b)] $a(n) \not= a(n+1)$ if and only if
$n \in \{ \lfloor \varphi k + 1/\varphi \rfloor \suchthat k \geq 0 \}$.
\item[(c)] $a(n) \not\in \{ a(n-1), a(n+1) \}$ if and only
if $n \in \{ \lfloor \varphi \lfloor \varphi^2 k \rfloor \rfloor \suchthat
k \geq 1 \}$.
\end{itemize}
\end{proposition}

\begin{proof}
\leavevmode
\begin{itemize}
\item[(a)]
We use the following {\tt Walnut} command:
\begin{verbatim}
eval twoconsec "?msd_fib An (Ex $a105774(n,x) & $a105774(n+1,x)) <=>
   (Ek,y (k>0) & $phi2n(k,y) & y=n+1)":
\end{verbatim}
\item[(b)]
We use the following {\tt Walnut} command:
\begin{verbatim}
eval differ "?msd_fib An (Ex,y $a105774(n,x) & $a105774(n+1,y) & 
   x!=y) <=> (Ek,y $phin(k+1,y) & y=n+1)":
\end{verbatim}
\item[(c)]
We use the following {\tt Walnut} commands:
\begin{verbatim}
def a003623 "?msd_fib Ex $phi2n(n,x) & $phin(x,z)":
eval isolated "?msd_fib An (n>0) => ((Ex,y,z $a105774(n-1,x) & 
   $a105774(n,y) & $a105774(n+1,z) & x!=y & y!=z) <=>
   (Ek $a003623(k,n)))":
\end{verbatim}
\end{itemize}
\end{proof}

The automaton computing \seqnum{A003623} is given in Figure~\ref{fig5}.
\begin{figure}[H]
\begin{center}
\includegraphics[width=6.5in]{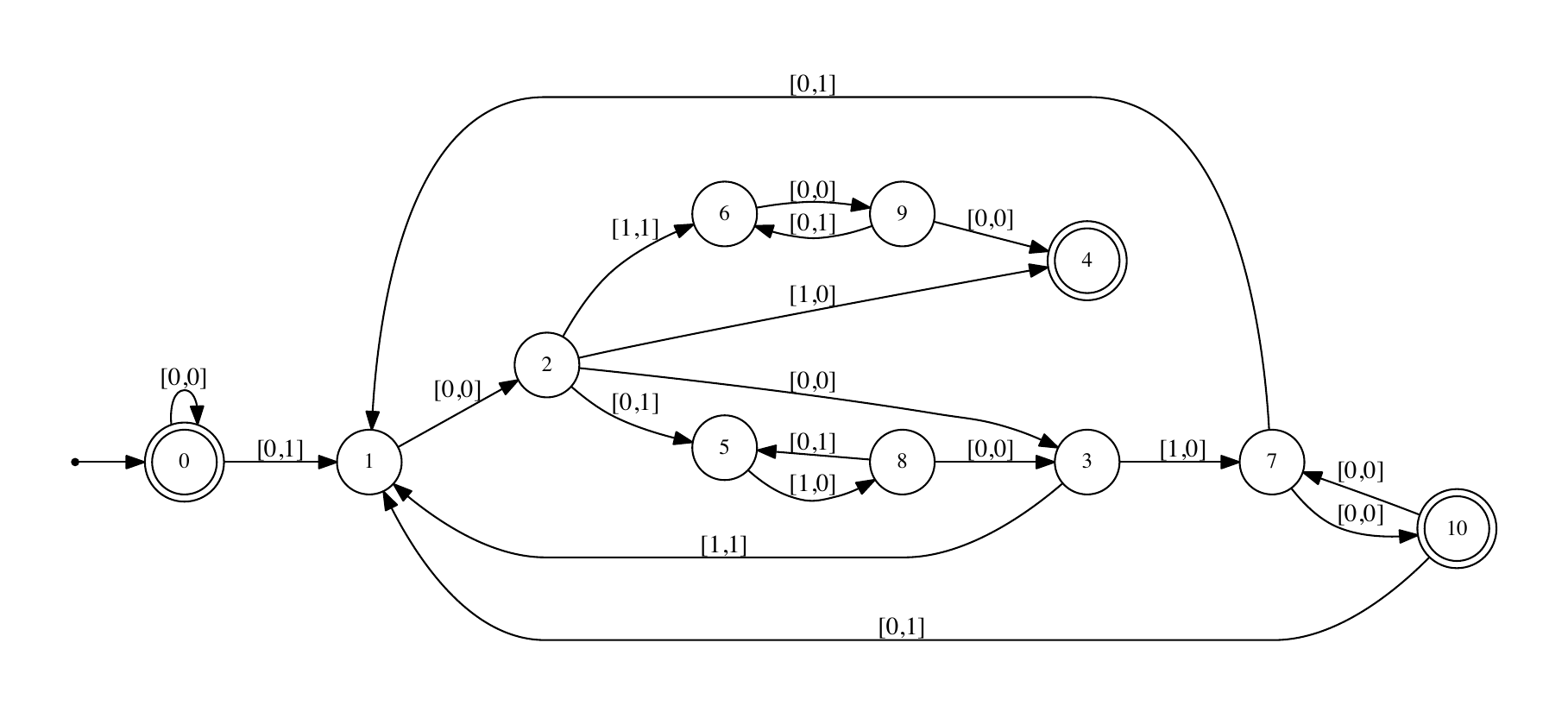}
\end{center}
\caption{Synchronized Fibonacci DFAO for \seqnum{A003623}.}
\label{fig5}
\end{figure}

\section{Sorting the terms}

We can consider sorting the terms of \seqnum{A105774} in ascending order.
The resulting sequence is \seqnum{A368200}.  We can now guess a Fibonacci
automaton for this sequence; it is displayed in Figure~\ref{fig6}.
\begin{figure}[H]
\begin{center}
\includegraphics[width=6.5in]{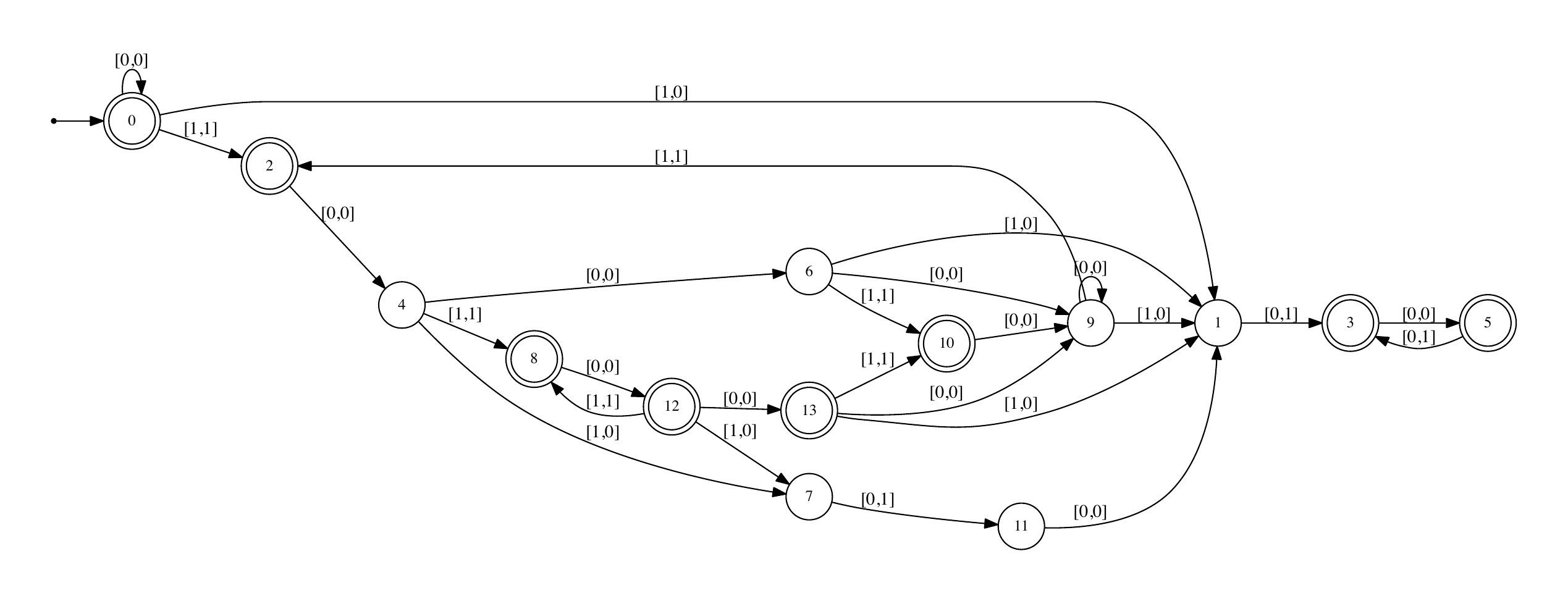}
\end{center}
\caption{Synchronized Fibonacci DFAO for \seqnum{A368200}.}
\label{fig6}
\end{figure}
It now remains to verify that this is indeed \seqnum{A105774} in ascending
order.  This is a consequence of a more general (and new) result, as
follows.
\begin{theorem}
Suppose $(a(n))_{n \geq 0}$ and $(b(n))_{n \geq 0}$
are synchronized sequences.  Then the following property is
decidable:   $(a(n))_{n \geq 0}$ is a permutation of $(b(n))_{n \geq 0}$ .
\label{new1}
\end{theorem}

\begin{proof}
We just sketch the proof, which depends on linear representations.
For more about these, see \cite[Chap.~9]{Shallit:2022} and
\cite{Berstel&Reutenauer:2011}.
Given automata for $(a(n))_{n \geq 0}$ and $(b(n))_{n \geq 0}$,
we can find linear representations counting the number of occurrences
of each natural number $n$ in each sequence.   Then we can form the
linear representation for the difference of these two.   Then we
use the fact that whether a linear representation is identically $0$
is decidable.
\end{proof}

When we perform this calculation
for the sequences \seqnum{A105774} and \seqnum{A368200},
we find that the linear representation for the difference is indeed
zero. It now simply remains to check
that the numbers computed by the automaton are in ascending order:
\begin{verbatim}
eval ascending "?msd_fib An,x,y ($a368200(n,x) & $a368200(n+1,y)) => y >= x":
\end{verbatim}

\begin{proposition}
Letting $b(n) = \seqnum{A368200}(n)$, we have
$b(n+1)-b(n) \in \{0,1,2 \}$.
\end{proposition}

\begin{proof}
We use the following {\tt Walnut} code.
\begin{verbatim}
def diff "?msd_fib Ex,y $a368200(n,x) & $a368200(n+1,y) & y=x+z":
def checkdiff "?msd_fib An,z $diff(n,z) => (z=0|z=1|z=2)":
\end{verbatim}
\end{proof}

It now follows that the first difference sequence $(d(n))_{n \geq 0}$
defined by $d(n) = b(n+1)-b(n)$ is Fibonacci-automatic.  
In fact, it is $c(n)$ in disguise, as the next theorem shows.
\begin{theorem}
For $n \geq 0$ we have $d(n) = 2-c(n)$.
\end{theorem}

\begin{proof}
We use the following {\tt Walnut} commands:
\begin{verbatim}
eval cd0 "?msd_fib An $diff(n,0) <=> C[n]=@2":
eval cd1 "?msd_fib An $diff(n,1) <=> C[n]=@1":
eval cd1 "?msd_fib An $diff(n,2) <=> C[n]=@0":
\end{verbatim}
and {\tt Walnut} returns {\tt TRUE} each time.
\end{proof}

\section{Special values}

In this section we compute the values of $a(n)$ when $n$ is a Fibonacci
or Lucas number.

\begin{theorem}
Let $s(n) = a(F_n)$.
\begin{itemize}
\item[(a)]  The Fibonacci representation of 
$s(n)$ for $n \geq 3$ is
$(1000)^{\lfloor (n-3)/4 \rfloor} x_n$ where 
$$x_n = \begin{cases}
	10, & \text{if $n \equiv \modd{0} {4}$;} \\
	101, & \text{if $n \equiv \modd{1} {4}$;} \\
	1001, & \text{if $n \equiv \modd{2} {4}$;}\\
	1, & \text{if $n\equiv \modd{3} {4}$.} 
	\end{cases} $$
\item[(b)]  $(s(n-1))$ is the OEIS sequence
\seqnum{A006498}, satisfying the recurrence $s(n) = s(n-1)+s(n-3)+s(n-4)$
for $n\geq 4$.
A closed form is as follows:
$$ s(n) = L_n/10 + F_n/2 + \begin{cases}
	-{1\over 5}(-1)^{n/2}, & \text{$n$ even;} \\
	{2\over 5}(-1)^{(n-1)/2}, & \text{$n$ odd.}
	\end{cases}$$
\end{itemize}
\label{fibb}
\end{theorem}

\begin{proof}
For part (a) we create a synchronized automaton that
accepts the Fibonacci representation of $F_n$,
which is $1 0^{n-2}$, in parallel with $s(n)$.
It is displayed in Figure~\ref{fig8}, from which
the result immediately follows by inspection.
\begin{verbatim}
reg isfib msd_fib "0*10*":
def special "?msd_fib $isfib(x) & $a105774(x,y)":
\end{verbatim}
\begin{figure}[H]
\begin{center}
\includegraphics[width=6.5in]{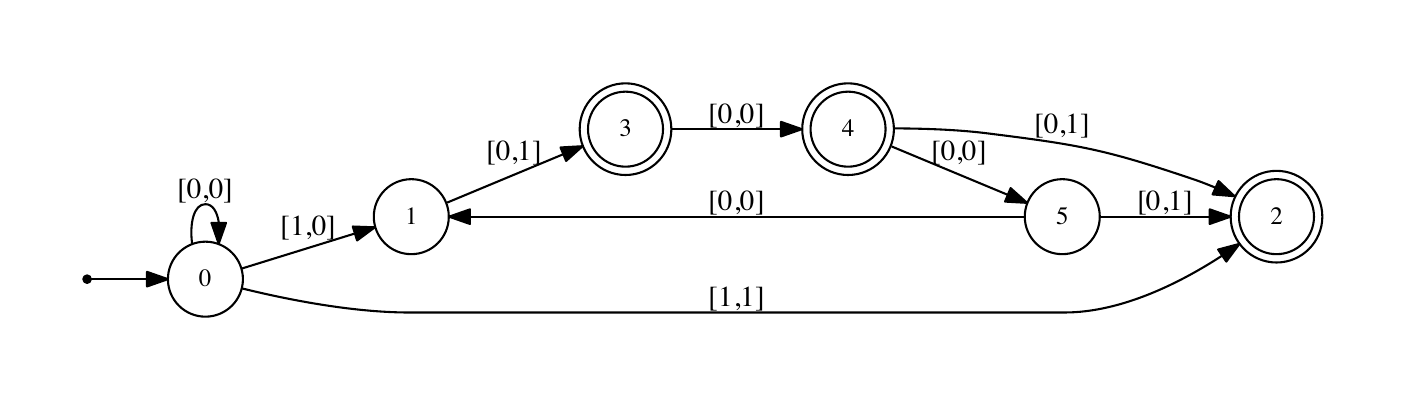}
\end{center}
\caption{Synchronized Fibonacci DFAO for $s(n)$.}
\label{fig8}
\end{figure}

For part (b) we first create a DFA that accepts, in
parallel, the Fibonacci representations of
$F_{n+4}, F_{n+3}, F_{n+1}, F_n$.
Then we check the recurrence.
\begin{verbatim}
reg four msd_fib msd_fib msd_fib msd_fib 
   "[0,0,0,0]*[1,0,0,0][0,1,0,0][0,0,0,0][0,0,1,0][0,0,0,1][0,0,0,0]*":
eval partb "?msd_fib Aa,b,c,d,x,y,z,w ($four(a,b,c,d) & $a105774(a,x) &
   $a105774(b,y) & $a105774(c,z) & $a105774(d,w)) => x=y+z+w":
\end{verbatim}
and {\tt Walnut} returns {\tt TRUE}.
\end{proof}

\begin{proposition}
\leavevmode
\begin{itemize}
\item[(a)] Suppose $k \geq 3$.
Over the range $F_k \leq n < F_{k+1}$, the
sequence $a(n)$ achieves its minimum uniquely at $n = F_k$.
\item[(b)] Suppose $k \geq 5$.  Over the range
$F_k \leq n < F_{k+1}$, the sequence
$a(n)$ achieves its maximum only at $n=F_k+1$ and $n=F_k+2$.
\end{itemize}
\end{proposition}

\begin{proof}
We use the following {\tt Walnut} code:
\begin{verbatim}
eval minval "?msd_fib Ax,y,z,t,u ($adjfib(x,y) & $a105774(x,z) &
   t>x & t<y & $a105774(t,u)) => u>z":
eval maxval "?msd_fib Ax,y,z,t,u,w (x>=5 & $adjfib(x,y) & $a105774(x+1,z) &
$a105774(x+2,w) & t>=x & t<y & $a105774(t,u)) => (z=w & u<=z)":
\end{verbatim}
and {\tt Walnut} returns {\tt TRUE} for both.
\end{proof}

\begin{corollary}
Suppose $F_k \leq n < F_{k+1}$.   Then
$s(k) \leq a(n) \leq F_{k+1} -1$.
\end{corollary}

\begin{theorem}
Define $t(n) = a(L_n)$ for $n \geq 0$.   Then
$t(n) = t(n-1)+t(n-3)+t(n-4)$ for $n \geq 5$.
A closed form for $n\geq 2$ is given by
$$ t(n) = 3L_n/10 + 3F_n/2 + \begin{cases}
	{2\over 5}(-1)^{n/2}, & \text{$n$ even;} \\
	{1\over 5}(-1)^{(n-1)/2}, & \text{$n$ odd.}
	\end{cases}$$
\end{theorem}

\begin{proof}
Same ideas as in Theorem~\ref{fibb}.  We omit the details.
\end{proof}

\section{Parity}

\begin{proposition}
We have $a(n) \equiv \modd{\lfloor n \varphi \rfloor} {2}$ for
$n \geq 0$.
\end{proposition}

\begin{proof}
We use the following {\tt Walnut} code:
\begin{verbatim}
def even "?msd_fib Ek n=2*k":
eval checkparity "?msd_fib An,x,y ($a105774(n,x) & $phin(n,y)) =>
   ($even(x) <=> $even(y))":
\end{verbatim}
and {\tt Walnut} returns {\tt TRUE}.
\end{proof}

\section{Distinct values}

Given a sequence taking values in $\Enn$ we can consider the same
sequence obtained by scanning the terms from left-to-right, and deleting
any values that previous occurred in the sequence.    We call this
the {\it distinctness transform} of the original sequence.
For example, the distinctness transform of
\seqnum{A105774} starts $0,1,2,4,7,6,12,11,\ldots$.
We have the following new theorem.
\begin{theorem}
The following problem is decidable:
given two synchronized automata computing $s(n)$ and
$s'(n)$, whether $(s'(n))$ is the distinctness
transform of $(s(n))$.
\label{new2}
\end{theorem}

\begin{proof}
It suffices to create a first-order logical formula asserting
that $(s'(n))$ is the distinctness transform of $(s(n))$.  

This formula asserts
\begin{itemize}
\item[(a)] An integer occurs in the range of $s(n)$ iff it occurs
in the range of $s'(n)$;
\item[(b)] All the values $s'(n)$ for $n \geq 0$ are distinct;
\item[(c)] For all $x$ and $y$, if $x$'s first occurrence in $(s(n))$
precedes that of $y$'s first occurrence, then $x$'s unique occurrence
in $(s'(n))$ also precedes that of $y$.
\end{itemize}
\end{proof}

Let us carry this out for \seqnum{A105774}.   Let $(a'(n))$ denote
the distinctness transform of $(a(n))$.  
The first few values are given in Table~\ref{tab10}.
\begin{table}[H]
\begin{center}
\begin{tabular}{c|ccccccccccccccccccccc}
$n$ & 0& 1& 2& 3& 4& 5& 6& 7& 8& 9&10&11&12&13&14&15&16&17&18&19 \\
\hline
$a'(n)$ & 0& 1& 2& 4& 7& 6&12&11& 9&20&19&17&14&15&33&32&30&27&28&22
\end{tabular}
\end{center}
\label{tab10}
\end{table}
We can guess the automaton for this sequence and then verify it.
\begin{figure}[H]
\begin{center}
\includegraphics[width=6.5in]{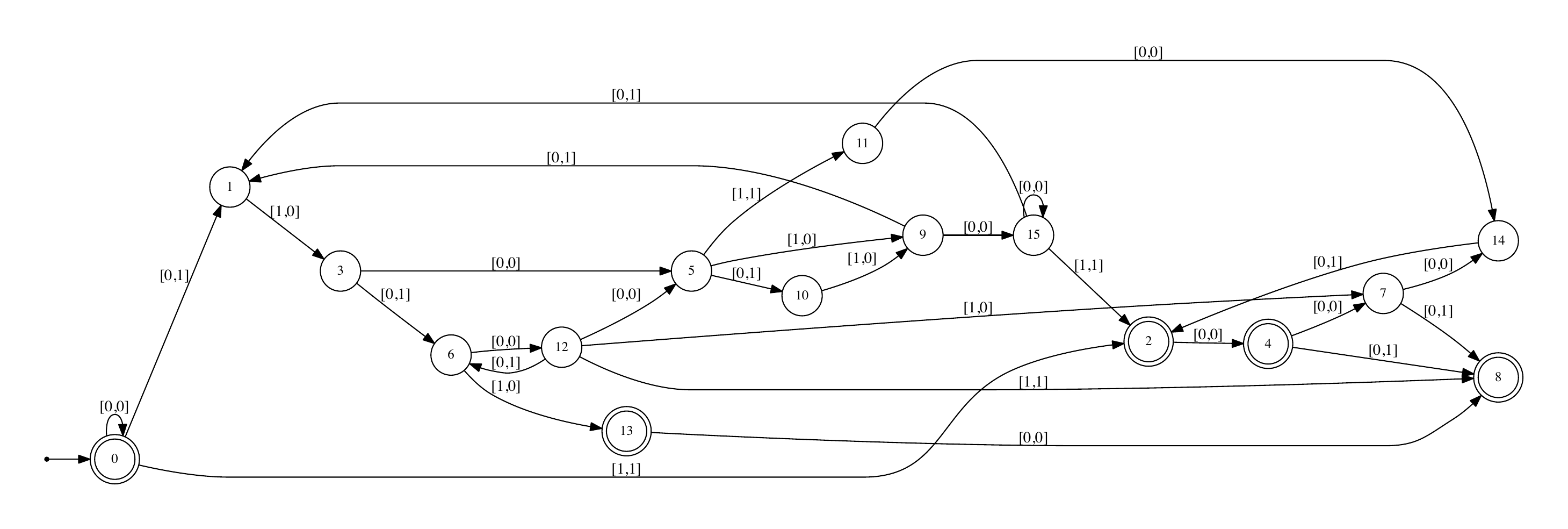}
\end{center}
\caption{Synchronized Fibonacci automaton for $a'(n)$.}
\label{fig10}
\end{figure}

First let us check that our guessed automaton really does compute a function:
\begin{verbatim}
def checkap1 "?msd_fib An Ex $aprime(n,x)":
def checkap2 "?msd_fib ~En,x1,x2 x1!=x2 & $aprime(n,x1) & $aprime(n,x2)":
\end{verbatim}
and {\tt Walnut} returns {\tt TRUE} for both.

Next, let's check the three conditions guaranteeing that $a'$ is the
distinctness transform of $a$:

\begin{verbatim}
eval check_distinct1 "?msd_fib Ax (Em $a105774(m,x)) <=> (En $aprime(n,x))":
eval check_distinct2 "?msd_fib ~En1,n2,x n1!=n2 & $aprime(n1,x) &
   $aprime(n2,x)":
def first_occ "?msd_fib $a105774(y,n) & Ax (x<y) => ~$a105774(x,n)":
# provided n occurs in A105774, gives the least position y 
# where it occurs
eval check_distinct3 "?msd_fib Ax,y,i,j ($first_occ(x,i) & $first_occ(y,j)
   & i<j) => Em,n $aprime(m,x) & $aprime(n,y) & m<n":
\end{verbatim}
and {\tt Walnut} returns {\tt TRUE} for all of them.

We now examine the run-length encoding of \seqnum{A105774}.  The {\it run-length
encoding\/} of an arbitrary sequence $(s(n))_{n \geq 0}$ counts the size of 
maximal blocks of consecutive identical values \cite{Golomb:1966a}.   For \seqnum{A105774}
it is $1,2,1,2,2,1,2,1,2,\ldots$.

Once we have the distinct terms of \seqnum{A105774}, we can easily
determine the run-length encoding of the sequence.
\begin{theorem}
The run-length encoding of \seqnum{A105774} is
\seqnum{A001468}, that is,
$$1 (2-{\bf f}[0]) (2-{\bf f}[1]) (2-{\bf f}[2]) \cdots$$
where ${\bf f} = 01001\cdots$ is the infinite Fibonacci word
\cite{Berstel:1986b}.
\end{theorem}
\begin{proof}
The claim can be verified by the following {\tt Walnut} code:
\begin{verbatim}
def nthrun2 "?msd_fib Ex,y $aprime(n,x) & $first_occ(x,y) & 
   $a105774(y+1,x)":
eval compare_fib "?msd_fib An $nthrun2(n+1) <=> F[n]=@0":
\end{verbatim}
\end{proof}
\section{Other results}
\label{other}

Define $w(n)$ to be the least integer such that $a(w(n)) \geq n$.
The first few terms of this sequence are given in Table~\ref{tab19}.
\begin{table}[H]
\begin{center}
\begin{tabular}{c|ccccccccccccccccccccc}
$n$ & 0& 1& 2& 3& 4& 5& 6& 7& 8& 9&10&11&12&13&14&15&16&17&18&19&20 \\
\hline
$w(n)$ & 0 & 1& 3& 4& 4& 6& 6& 6& 9& 9& 9& 9& 9&14&14&14&14&14&14&14&14
\end{tabular}
\end{center}
\label{tab19}
\end{table}

\begin{proposition}
If $n\geq 2$ and $F_k \leq n < F_{k+1}$ then $w(n) = F_k + 1$.
\end{proposition}

\begin{proof}
We use the following {\tt Walnut} commands:
\begin{verbatim}
def trapfib2 "?msd_fib $adjfib(x,y) & x<=k & y>k":
def wseq "?msd_fib (Em $a105774(x,m) & m>=n) & 
   (Ai,p (i<x & $a105774(i,p)) => p<n)":
eval propw "?msd_fib Ax,y,n,m (n>=2 & $trapfib2(n,x,y) & $wseq(n,m))
   => m=x+1":
\end{verbatim}
and {\tt Walnut} returns {\tt TRUE}.
\end{proof}

We now turn to describing the fixed points of $a$.
\begin{theorem}
We have $a(n)=n$ for $n>0$ if and only if
$(n)_F \in 1 (00100^*1)^* \{ \epsilon,01,010,0100 \}$.
\end{theorem}
\begin{proof}
We use the {\tt Walnut} command
\begin{verbatim}
def fixed "?msd_fib $a105774(n,n)":
\end{verbatim}
and it produces the automaton in Figure~\ref{fig20}, from
which we can directly read off the result.
\begin{figure}[H]
\begin{center}
\includegraphics[width=6.5in]{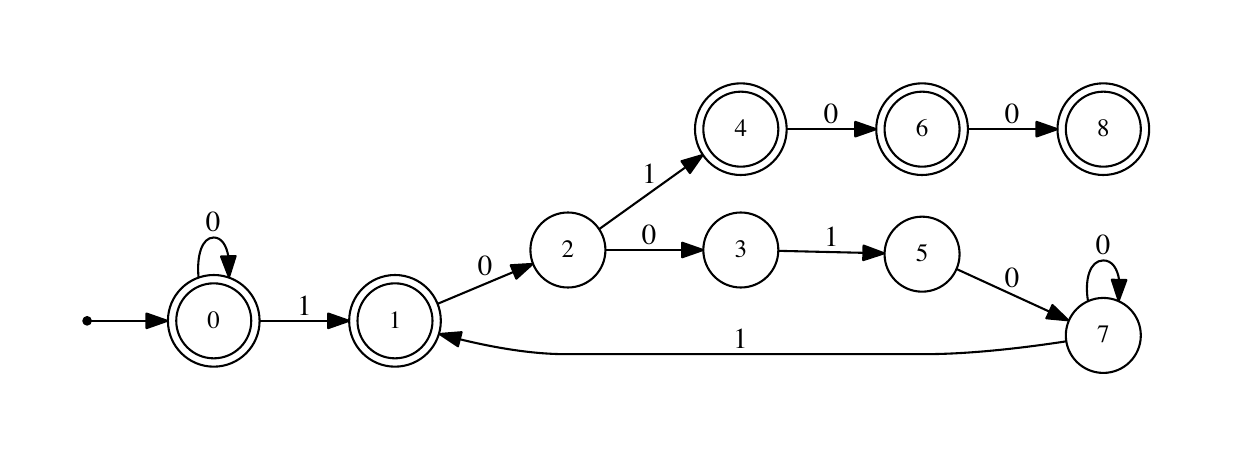}
\end{center}
\caption{Solutions to $a(n) = n$.}
\label{fig20}
\end{figure}
\end{proof}

For the remainder of this section, let $a(n)$ represent (as before)
the $n$'th term of \seqnum{A105774}, 
let $b(n) = \lfloor n \varphi \rfloor$,
and let $c(n) = \lfloor n \varphi^2 + 1/2 \rfloor$.
We will find some relations concerning the composition of the functions
$a$, $b$, and $c$.  We introduce the function
$x(n) := a(b(n)) - b(a(n))$.

\begin{theorem}
\leavevmode
\begin{itemize}
\item[(a)] $a(b(n)) \geq b(a(n))$ for all $n \geq 0$.
\item[(b)] For $n \geq 1$,
the difference $x(n) := a(b(n)) - b(a(n))$ is exactly the parity of the number of $1$'s
in the Fibonacci representation of $n-1$, that is, sequence
\seqnum{A095076}.
\item[(c)] $b(a(b(n))) \geq a(b(a(n)))$ for all $n \geq 0$.
\item[(d)] $a(a(b(n)))-a(b(a(n))) \in \{-2,0,2\}$ for all $n \geq 0$.
\item[(e)] If $d(n) := c(a(n)) - a(b(n)) - a(n)$, then
$d(n) \in \{ -1,0, 1 \}$ for all $n \geq 0$.
\item[(f)] For $n\geq 1$ we have $d(n) = 0$ iff ${\bf f}[n-1]=1$.
\item[(g)] For $n \geq 1$ we have
$c(a(b(n)))+a(b(n))-a(b(b(n)))-2b(a(n)) = 1$.
\end{itemize}
\end{theorem}

\begin{proof}
We use the following {\tt Walnut} commands:
\begin{verbatim}
reg even1 msd_fib "(0*10*1)*0*":
# even number of 1's in Fibonacci representation

def ab "?msd_fib Ex $phin(n,x) & $a105774(x,z)":
# z = a105774(floor(phi*n))

def ba "?msd_fib Ex $a105774(n,x) & $phin(x,z)":
# z = floor(phi*a105774(n))

# check that ab >= ba for all n so subtraction makes sense over N:
def test "?msd_fib An,x,y ($ab(n,x) & $ba(n,y)) => x>=y":
# Walnut says TRUE

def xx "?msd_fib Ex,y $ab(n,x) & $ba(n,y) & z=x-y":
# 12 states

#test that a(b(n)) = b(a(n)) precisely when parity of 1's in
#Fibonacci expansion of n-1 is 0:

eval test1 "?msd_fib An $xx(n+1,0) <=> $even1(n)":
# Walnut says TRUE
#so it follows immediately
# that a(b(n)) > b(a(n)) precisely when parity of 1's in
#Fibonacci expansion of n-1 is 1, no need to test it

def aba "?msd_fib Ex $ba(n,x) & $a105774(x,z)":
# 38 states
def bab "?msd_fib Ex $ab(n,x) & $phin(x,z)":
# 18 states

# check the inequality bab(n) >= aba(n):
eval test3 "?msd_fib An,x,y ($bab(n,x) & $aba(n,y)) => x>=y":
# Walnut says true

def aab "?msd_fib Ex $ab(n,x) & $a105774(x,z)":
# 34 states

# check the proposition that aab(n) - aba(n) takes values in
# {-2,0,2} only:
eval test4 "?msd_fib An,x,y ($aab(n,x) & $aba(n,y)) => (x=y|x=y+2|y=x+2)":
# Walnut says TRUE

def a004937 "?msd_fib Ex $phi2n(2*n,x) & z=(x+1)/2":
def ca "?msd_fib Ex $a105774(n,x) & $a004937(x,z)":
def dp "?msd_fib Ew,x,y $ca(n,w) & $ab(n,x) & $a105774(n,y) & z+x+y=w+1":
# here dp computes d(n)+1
eval test1 "?msd_fib An,x $dp(n,x) => (x=0|x=1|x=2)":
eval test2 "?msd_fib An (n>=1) => (F[n-1]=@1 <=> $dp(n,1))":

def cab "?msd_fib Ex $ab(n,x) & $a004937(x,z)":
def abb "?msd_fib Ex $phin(n,x) & $ab(x,z)":
eval test3 "?msd_fib An,r,s,t,u (n>=1 & $cab(n,r) & $abb(n,s) & $ab(n,t) 
   & $ba(n,u)) => r+t=s+2*u+1":
\end{verbatim}
\end{proof}

\section{Carlitz-style results}

In a classic paper \cite{Carlitz&Scoville&Hoggatt:1972}, Carlitz et al.\ obtained results involving compositions of the functions
$b(n) = \lfloor \varphi n \rfloor$ and
$d(n) = \lfloor \varphi^2 n \rfloor$.   We can obtain a version of
their Theorem 13, involving 
$a(n) = \seqnum{A105774}(n)$,
as follows:\footnote{We warn the reader that our $b(n)$
is written as $a(n)$ in the paper of Carlitz et al., and our $d(n)$ is their
$b(n)$.}

We adopt the following notation:   if $u = u_1 \cdots u_t \in
\{ a, b,d,x\}^*$, then by $u(n)$ we mean $(u_1 \circ \cdots \circ u_t)(n)$.
Recall that $x$ was defined in Section~\ref{other}.
\begin{theorem}
Let $u \in \{ b,d \}^+$ be a nonempty word with $i$ copies of $b$ and
$j$ copies of $d$.   Then for all $n \geq 1$ we have
$$ au(n) = F_{i+2j} \cdot ab(n) + F_{i+2j-1} \cdot a(n) + C(u) \cdot (2x(n) - 1),$$
where $C(u)$ is defined as follows:
\begin{equation}
C(u) = \begin{cases}
	0, & \text{if $u = b$}; \\
	1, & \text{if $u = d$}; \\
	F_{i+2j-1} + C(v), & \text{if $u = vb$}; \\
	F_{i+2j-1} - C(v), & \text{if $u = vd$}.
	\end{cases}
\label{cu}
\end{equation}
\end{theorem}

\begin{proof}
First we need some lemmas:
\begin{lemma}
We have
\begin{itemize}
\item[(a)] $xb(n) = x(n)$ for all $n \geq 0$.
\item[(b)] $xd(n) = 1-x(n)$ for all $n \geq 1$.
\item[(c)] $abb(n) = ab(n) + a(n) + 2x(n) - 1$ for all $n \geq 1$.
\item[(d)] $abb(n) = ad(n)$ for all $n \geq 0$.
\item[(e)] $abd(n) = 2ab(n) + a(n) + 2x(n) - 1$ for all $n \geq 1$.
\end{itemize}
\end{lemma}

\begin{proof}
These can be proved with {\tt Walnut} as follows.   Recall that
{\tt phin} is {\tt Walnut}'s name for $b$ and {\tt phi2n} is {\tt Walnut}'s
name for $d$.
\begin{verbatim}
eval checka "?msd_fib An,y,z,w (n>=0 & $xx(n,y) & $phin(n,z) & $xx(z,w)) 
   => w=y":
eval checkb "?msd_fib An,y,z,w (n>=1 & $xx(n,y) & $phi2n(n,z) & $xx(z,w)) 
   => w+y=1":
eval checkc "?msd_fib An,y,z,w,t (n>=1 & $abb(n,y) & $ab(n,z) & 
   $a105774(n,w) & $xx(n,t)) => y+1=z+w+2*t":
def ad "?msd_fib Ew $phi2n(n,w) & $a105774(w,z)":
def abd "?msd_fib Ew,y $phi2n(n,w) & $phin(w,y) & $a105774(y,z)":
eval checkd "?msd_fib An,z (n>=0 & $ad(n,z)) => $abb(n,z)":
eval checke "?msd_fib An,y,z,w,t (n>=1 & $abd(n,y) & $ab(n,z) & 
   $a105774(n,w) & $xx(n,t)) => y+1=2*z+w+2*t":
\end{verbatim}
and {\tt Walnut} returns {\tt TRUE} for all.
\end{proof}

Now we can prove the desired result by induction on $|u|$.   
It is clear from the above results for $|u|=1$.  Now assume $|u|\geq 2$
and consider two cases:  $u = vb$ and $u = vd$.

$u=vb$:  In this case we have, by induction, that
$$ av(n) = F_{i+2j-1} \cdot ab(n) + F_{i+2j-2} \cdot a(n) + C(v) \cdot (2x(n) - 1).$$
Substituting $b(n)$ for $n$, we get
\begin{align*}
avb(n) &= F_{i+2j-1} \cdot abb(n) + F_{i+2j-2} \cdot ab(n) + C(v) \cdot (2xb(n) - 1) \\
&= F_{i+2j-1} \cdot (ab(n)+a(n) + 2x(n)-1) + F_{i+2j-2} ab(n) + C(v) \cdot (2xb(n) - 1) \\
&= F_{i+2j} \cdot ab(n) + F_{i+2j-1} a(n) + (C(v) + F_{i+2j-1}) (2x(n) - 1),
\end{align*}
as desired.

$u=vd$:  By induction we have
$$ av(n) = F_{i+2j-2} \cdot ab(n) + F_{i+2j-3} \cdot a(n) + C(v) \cdot (2x(n) - 1).$$
Substituting $d(n)$ for $n$, we get
\begin{align*}
avd(n) &= F_{i+2j-2} \cdot abd(n) + F_{i+2j-3} \cdot ad(n) + C(v) \cdot (2xd(n) - 1) \\
&= F_{i+2j-2} \cdot (2ab(n) + a(n) + 2x(n) - 1) +
F_{i+2j-3} (ab(n) + a(n) + 2x(n) - 1) \\
& \quad + C(v) (2(1-x(n)) - 1) \\
&= F_{i+2j} \cdot ab(n) + F_{i+2j-1} a(n) + (F_{i+2j-1} - C(v)) (2x(n)-1) .
\end{align*}
This completes the proof.
\end{proof}

We can obtain a linear representation for $C(u)$, as follows.
\begin{corollary}
Define vectors $L, R$ and a matrix-valued morphism $\mu$ as follows:
\begin{align*}
L &= \left[ \begin{array}{ccc}
	0 & 0 & 1 
	\end{array} \right] \\
\mu(b) &= \left[ \begin{array}{ccc}
	0 & 1 & 0 \\
	1 & 1 & 0 \\
	0 & 1 & 1 
	\end{array} \right] \\
\mu(d) &= \left[ \begin{array}{ccc}
	0 & 0 & 1 \\
	0 & 1 & 1 \\
	1 & 2 & 1 
	\end{array} \right]\\
R &= \left[ \begin{array}{c}
	1 \\
	0 \\
	0 \\
	\end{array}
	\right] .
\end{align*}
Then $C(u) = L \mu(u) R$.
\end{corollary}

\begin{proof}
We first obtain the following relations, for $v \in \{ b,d \}^+$:
\begin{align*}
C(vbb) &= C(v) + C(vb) + C(vd) \\
C(vbd) &= C(vd) \\
C(vdb) &= C(vb) + 2 C(vd) \\
C(vdd) &= C(v) + C(vb) + C(vd) .
\end{align*}
We just show how to obtain the first relation; the others are similar.
Letting $u = vbb$, and using \eqref{cu}, we get
\begin{align*}
C(u) &= C(vbb) = F_{i+2j-1} + C(vb) \\
C(vb) &= F_{i+2j-2} + C(v) 
\end{align*}
from which it follows that $C(u) = F_{i+2j} + C(v)$.
On the other hand, from \eqref{cu}, we have
$C(vd) = F_{i+2j-1} - C(v)$ 
and so
$$C(v) + C(vb) + C(vd) = C(v) + F_{i+2j-2} + C(v) + F_{i+2j-1} - C(v) 
= C(v) + F_{i+2j},$$ and the first identity follows.

Once we have the relations, we can combine them in the usual way to
obtain the given linear representation.
\end{proof}

\section{Going further}

It is possible to study many similar recurrences using exactly the
same techniques.   We list just three possibilities; many more should
occur to the interested reader.

\begin{enumerate}
\item One can study a variation of the defining recurrence
\eqref{eq1} where the Lucas numbers replace the Fibonacci numbers.
The resulting sequence is \seqnum{A118287} in the OEIS, and is computed
by a $102$-state synchronized automaton.  

With this automaton one can easily prove observations such as\\
\centerline{\seqnum{A118287}$(n)$ = \seqnum{A118287}$(n-2)$ if and only
if $n$ belongs to \seqnum{A003231}.}




\item One can generalize \eqref{eq1} by introducing two integer
parameters, $x$ and $y$, as follows:
\begin{equation}
a_{x,y}(n) = \begin{cases}
        n, & \text{if $n\leq 1$}; \\
        F_{j+1}x - a_{x,y}(n-F_j)y, & \text{if $F_j < n \leq F_{j+1}$ for
	$j \geq 2$}.
        \end{cases}
        \label{eq11}
\end{equation}
Equation \eqref{eq1} corresponds to $(x,y) = (1,1)$.  Numerical
experiments suggest that if $y = 1$, the resulting sequence
might be Fibonacci-synchronized for each $x \geq 1$.

\begin{itemize}
\item
Consider the particular case $(x,y) = (2,1)$.   Then one can prove
that $a_{2,1}(n)$ is computed by a $22$-state synchronized automaton.
Furthermore, $((a_{2,1} (n)-1)/2)_{n \geq 1}$ is
sequence \seqnum{A343647} and is a permutation of the positive integers.

\item
The case $(x,y) = (3,2)$ seems harder to analyze.  It is quite possible
that it is not Fibonacci-synchronized.
\end{itemize}

\item Another variation is the following:
\begin{equation}
b(n) = \begin{cases}
	n, & \text{if $n\leq 1$}; \\
	F_{j+1} - b(b(n-F_j)), & \text{if $F_j < n \leq F_{j+1}$ 
	for $j \geq 2$}.
	\end{cases}
	\label{eq12}
\end{equation}
This sequence is generated by a $24$-state synchronized automaton.

\end{enumerate}

\end{document}